\pdfoutput=1
\documentclass[graybox]{svmult}
\usepackage{graphicx}
\usepackage[all]{xy}
\usepackage{helvet,courier,type1cm}
\usepackage{amssymb}
\usepackage[left=3.5cm,top=3cm,bottom=3cm,right=3.5cm,nofoot]{geometry}

\graphicspath{%
{/home/rvb/MyDocs/Articles/Hamming/}
{/home/rvb/lisp/ga/data/}
{/home/roman/MyDocs/Articles/Hamming/}
{/home/roman/lisp/ga/data/}
{/home/USER/MyDocs/Articles/Hamming/}
{/home/USER/lisp/ga/data/}
}

\newcommand{\bR}{\mathbb{R}}

\newcommand{\bE}{\mathbb{E}}
\newcommand{\bN}{\mathbb{N}}

\newcommand{\cl}{\mathrm{cl}\,}
\newcommand{\co}{\mathrm{co}\,}
\newcommand{\ext}{\mathrm{ext}\,}
\newcommand{\cP}{\mathcal{P}}

\begin{document}

%\pagestyle{headings}  % switches on printing of running heads

%\addtocmark{} % additional mark in the TOC
%\chapter*{Preface}

%\mainmatter              % start of the contributions

\title*{Asymmetry of Risk and Value of Information}
\titlerunning{Asymmetry of Risk and Value of Information}
\author{Roman V. Belavkin}
\authorrunning{Roman Belavkin}
% \tocauthor{Roman Belavkin}
\institute{Roman V. Belavkin \at Middlesex University, London NW4 4BT, UK,
\email{R.Belavkin@mdx.ac.uk}}

\maketitle

%\thispagestyle{fancy}

%\lhead{}
%\chead{\footnotesize To appear in ... eds.\\{\em Dynamics of Information Systems: Theory and Applications}, Springer.}
%\chead{To appear in {\em Dynamics of Information Systems: Theory and Applications}, Springer.}
%\rhead{} \lfoot{} \cfoot{} \rfoot{}
%\renewcommand{\headrulewidth}{.4pt}

\abstract{
The von Neumann and Morgenstern theory postulates that rational choice under uncertainty is equivalent to maximization of expected utility (EU).  This view is mathematically appealing and natural because of the affine structure of the space of probability measures.  Behavioural economists and psychologists, on the other hand, have demonstrated that humans consistently violate the EU postulate by switching from risk-averse to risk-taking behaviour.  This paradox has led to the development of descriptive theories of decisions, such as the celebrated prospect theory, which uses an $S$-shaped value function with concave and convex branches explaining the observed asymmetry.  Although successful in modelling human behaviour, these theories appear to contradict the natural set of axioms behind the EU postulate.  Here we show that the observed asymmetry in behaviour can be explained if, apart from utilities of the outcomes, rational agents also value information communicated by random events.  We review the main ideas of the classical value of information theory and its generalizations.   Then we prove that the value of information is an $S$-shaped function, and that its asymmetry does not depend on how the concept of information is defined, but follows only from linearity of the expected utility.  Thus, unlike many descriptive and `non-expected' utility theories that abandon the linearity (i.e. the `independence' axiom), we formulate a rigorous argument that the von Neumann and Morgenstern rational agents should be both risk-averse and risk-taking if they are not indifferent to information.
\\
\\
{\bf Keywords}: Decision-making $\cdot$ Expected utility $\cdot$ Prospect theory $\cdot$ Uncertainty $\cdot$ Information}

\section{Introduction}
\label{sec:intro}

A theory of decision-making under uncertainty is extremely important, because it suggests models of rational choice used in many practical applications, such as optimization and control systems, financial decision-support systems and economic policies.  Therefore, the fact that one of the most fundamental principles of such a theory remains disputed for more than half a century is not only intriguing, but points at a lack of understanding with potentially dangerous consequences.  The principle is the von Neumann and Morgenstern expected utility postulate \cite{Neumann-Morgenstern}, which follows very naturally from some fundamental ideas of probability theory, and it has become an essential part of game theory, operations research, mathematical economics and statistics (e.g. \cite{Wald50,Savage54}).  Several researchers, however, were sceptical about the validity of the postulate, and devised clever counter-examples undermining the expected utility idea (e.g. \cite{Allais53,Ellsberg61}).  Psychologists and behavioural economists have studied such examples in experiments and demonstrated consistently over several decades that the expected utility fails to explain human behaviour in some situations of making choice under uncertainty (e.g. see \cite{Tversky81,Huck-Muller12}).  The attempts to dismiss these observations simply by humans' ignorance about game and probability theories were quickly challenged, when professional traders were shown to conform to these `irrational' patterns of decision-making \cite{List-Haigh05}.  A suggestion that the human mind is somehow inadequate for making decisions under uncertainty should be taken with caution, considering that it has evolved over millions of years to do exactly that.

\begin{figure}[t]
\begin{center}
\input{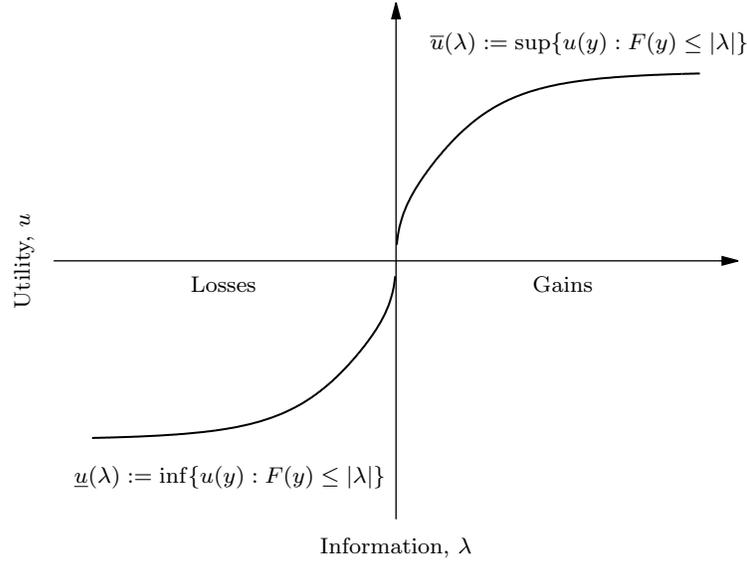}
\end{center}
\caption{An $S$-shaped value function with concave and convex branches used in prospect theory \cite{Kahneman-Tversky79} to model risk-aversion for gains and risk-taking for losses.  These properties also characterize two branches of the value of information: $\overline u(\lambda)$ is concave and plotted here against `positive' information associated with gains; $\underline u(\lambda)$ is convex and plotted against `negative' information associated with losses.}
\label{fig:ui-s}
\end{figure}

One of the most successful behavioural theories explaining the phenomenon is prospect theory \cite{Kahneman-Tversky79}, which suggests that humans value prospects of gains differently from prospects of losses, and therefore their attitude to risk is different in these situations.  To model this asymmetry of risk an $S$-shaped value function with concave and convex branches was proposed (e.g. see Figure~\ref{fig:ui-s}).  Unfortunately, it is precisely this asymmetry that appears to be in conflict with the expected utility theory, and specifically with the axioms that imply its linear (or affine) properties (the so-called `independence' axiom \cite{Machina82}).  Many attempts to develop theories without such axioms have been made, such as the regret theory \cite{Loomes-Sugden82:_regret} and other `non-expected' utility theories (see \cite{Starmer2000,Machina03,Machina04}).  The main aim of this work is to show that another approach is possible, and it involves one important concept emerging from physics and now entering new areas of science, and it is the concept of \emph{entropy}.

Entropy is an information potential, and decision-making under uncertainty can be improved, if some additional information is provided.  This improvement implies that information has utility, and the amalgamation of these two concepts is known as the value of information theory, which was developed in the mid 1960s by Stratonovich and Grishanin as a branch of information theory and theoretical cybernetics \cite{Stratonovich65,Stratonovich-Grishanin66,Stratonovich66:_value_automata,Grishanin-Stratonovich66,Stratonovich67:_inf_dyn,Stratonovich-Grishanin68,Stratonovich75:_inf}.  This theory considered variational problems of maximization or minimization of expected utility subject to constraints on information.  One of many interesting results is an $S$-shaped value function representing the value of information, which resembles the $S$-shaped value function in prospect theory.  Analysis shows that this geometric property is the consequence of linearity of the expected utility functional, and it is independent of any specific definition of information \cite{Belavkin11:_optim}.  Thus, rational agents that are not indifferent to information should value information about gains differently from information about losses, and this may explain the observed asymmetry in humans' attitude towards risk.  The advantage of the proposed approach is that it does not contradict, but generalizes the expected utility postulate.

In the next section, we review the main mathematical principles behind the expected utility postulate.  The presentation of axioms follows the theory of ordered vector spaces, and it allows the author to give a very short and simple proof of the postulate in Theorem~\ref{th:linear-utility}.  The aim of this section is to show that the ideas behind the expected utility are very natural and fundamental.  Section~\ref{sec:risk} overviews several classical examples that are often used in psychological experiments to test humans' preferences and attitude towards risk.  Some examples are presented in a slightly simplified form to illustrate the idea.  The basic concepts of information theory and the classical value of information theory are presented in the first half of Section~\ref{sec:u-of-i}.  Then an abstraction will be made using convex analysis to show that the $S$-shape characterizes the value of an abstract information functional.  We conclude with a brief discussion of the paradoxes.

\section{Linear Theory of Utility}
\label{sec:linear}

We review the definition of a preference relation, its utility representation and the condition of its existence.  Then we show that in the category of linear spaces, such as the vector space of measures, the preference relation should be linear and represented by a linear functional, such as the expected utility.

\subsection{Abstract Choice Sets and Their Representations}

A set $\Omega$ is called an abstract {\em choice} set, if any pair of its elements can be compared by a transitive binary relation $\lesssim$, called the \emph{preference} relation:
\begin{definition}[Preference relation]
A binary relation $\lesssim\subseteq\Omega\times\Omega$ that is
\begin{enumerate}
\item Total\footnote{This property is sometimes called \emph{completeness}, but this term often has other meanings in order theory (e.g. complete partial order) or topology (e.g. complete metric space).}: $a\lesssim b$ or $a\gtrsim b$ for all $a$, $b\in\Omega$.
\item Transitive: $a\lesssim b$ and $b\lesssim c$ implies $a\lesssim c$.
\end{enumerate}
\end{definition}
One can see that $\lesssim$ is a total \emph{pre-order} (reflexivity of $\lesssim$ follows from the fact that it is total).  We shall denote by $\gtrsim$ the inverse relation $(\lesssim)^{-1}$.  We shall distinguish between the strict and non-strict preference relations, which are defined respectively as follows:
\begin{eqnarray*}
a<b&:=&(a\lesssim b)\wedge \neg\,(a\gtrsim b)\\
a\sim b&:=&(a\lesssim b)\wedge (a\gtrsim b)
\end{eqnarray*}
Non-strict preference $\sim$ is also called an {\em indifference}, and it is an equivalence relation.  The quotient set $\Omega/\sim$ defined by this equivalence relation is the set of equivalence classes $[a]:=\{b\in\Omega:a\sim b\}$, which are totally ordered.

It is quite natural in applications to map the choice set to some standard ordered set, such as $\bN$ or $\bR$.  Such numerical mapping is called a \emph{utility representation}:
\begin{definition}[Utility representation of $\lesssim$]
A real function $u:(\Omega,\lesssim)\to(\bR,\leq)$ such that:
\[
a\lesssim b\quad\iff\quad u(a)\leq u(b)
\]
\end{definition}
Observe that the mapping above is monotonic in both directions, which means that utility defines an order-embedding of $(\Omega/\sim,\leq)$ into $(\bR,\leq)$.  Clearly, a utility representation exists for any countable choice set $\Omega$.  For uncountable $\Omega$, the existence of a utility representation is not guaranteed, and it is given by the following condition:
\begin{theorem}[Debreu \cite{Debreu54}]
A utility representation of uncountable $(\Omega,\lesssim)$ exists if and only if there is a countable subset $Q\subset\Omega$ that is {\em order dense}: for all $a<b$ in $\Omega\setminus Q$ there is $q\in Q$ such that $a<q<b$.
\label{th:order-dense}
\end{theorem}

Note that in optimization theory and its applications one often begins the analysis with a given real objective function $u:\Omega\to\bR$ (e.g. a utility function $u$ or a cost function $-u$).  The preference relation $\lesssim$ is then induced on $\Omega$ by the values $u(\omega)\in\bR$ as a pull-back of order $\leq$ on $\bR$.  This \emph{nuclear} binary relation $\lesssim$ is clearly total and transitive.  Therefore, although some works consider non-total or non-transitive preferences, as well as relations without a utility function, this paper focuses only on choice sets with utility representations.

\subsection{Choice Under Uncertainty}

By definition, a utility representation $u:\Omega\to\bR$ is an embedding of the pre-ordered set $(\Omega,\lesssim)$ into $(\bR,\leq)$, so that the quotient set $(\Omega/\sim,\leq)$ is order-isomorphic to the subset $u(\Omega)\subseteq\bR$.  Recall, however, that $(\bR,\leq)$ is more than just an ordered set --- it is a totally ordered field, in which the order is compatible with the algebraic operations of addition and multiplication, is Archimedian (see below), and it is the only such field.  Suppose that the choice set $\Omega$ is also equipped with some algebraic operations.  Then it appears quite natural if utility $u:\Omega\to\bR$ is compatible also with these algebraic operations, acting as an algebraic isomorphism.  In the language of category theory, utility should be a morphism between objects $\Omega$ and $u(\Omega)\subseteq\bR$ of the same category.  For example, if $\Omega$ is a subset of a real vector space $Y$, then in the category of linear spaces or algebras, like $(\bR,\leq)$, pre-order $(Y,\lesssim)$ (extended from $\Omega\subseteq Y$) should be compatible with the vector space operations
\begin{eqnarray}
x\lesssim y&\iff& \lambda x\lesssim\lambda y\,,\qquad\qquad\forall\,\lambda>0\label{ax:lambda}\\
x\lesssim y&\iff& x+z\lesssim y+z\,,\qquad\forall\,z\in Y\label{ax:addition}
\end{eqnarray}
and Archimedian
\begin{equation}
nx\lesssim y\,,\quad\forall\,n\in\bN\quad\Rightarrow\quad x\lesssim0
\label{ax:Archimedian}
\end{equation}
These three axioms are often assumed in the category of pre-ordered vector spaces.  Note that classical texts on expected utility (e.g. \cite{Neumann-Morgenstern,Savage54}) present these axioms in a different form, because of a restriction to an affine subspace of a vector space due the normalization and positivity conditions for probability measures.  Thus, axioms~(\ref{ax:lambda}) and (\ref{ax:addition}) are combined into the so-called \emph{independence} axiom:
\[
x\lesssim y\ \iff\ \lambda x+(1-\lambda)z\lesssim\lambda y+(1-\lambda)z\,,\qquad\forall\,z\in Y\,,\ \lambda\in(0,1]
\]
The Archimedian axiom~(\ref{ax:Archimedian}) is replaced by the \emph{continuity} axiom:
\[
x\lesssim y\lesssim z\quad\Rightarrow\quad y\sim\lambda x+(1-\lambda)z\,,\qquad\exists\,\lambda\in[0,1]
\]
The author finds it more convenient to work in the category of linear spaces and making the restriction to an affine subspace when necessary.   Thus, we shall assume axioms~(\ref{ax:lambda}), (\ref{ax:addition}) and (\ref{ax:Archimedian}).  Substituting $z=-x-y$ into~(\ref{ax:addition}) gives also
\begin{equation}
x\lesssim y\quad\iff\quad-x\gtrsim-y
\label{eq:negative}
\end{equation}
and together with axiom~(\ref{ax:lambda}) this property also implies that
\begin{equation}
x\sim y\quad\iff\quad\lambda x\sim\lambda y\,,\qquad\forall\,\lambda\in\bR
\label{eq:equivalence}
\end{equation}

The linear or affine algebraic structures occur naturally in measure theory and probabilistic models of uncertainty.  Indeed, consider a probability space $(\Omega,\mathcal{F},P)$, where $\Omega$ is the set of elementary events, $\mathcal{F}\subseteq 2^\Omega$ is a $\sigma$-algebra of events, and $P:\mathcal{F}\to[0,1]$ is a probability measure.  In the context of game theory or economics, the probability measure $P$, defined over the choice set $(\Omega,\lesssim)$ with utility $u:\Omega\to\bR$, is often referred to as a \emph{lottery}, emphasizing the fact that utility is now a random variable (assuming it is $\mathcal{F}$-measurable).  The \emph{expected utility} associated with event $E\subseteq\Omega$ is given by the integral:
\[
\bE_P\{u(E)\}:=\int_Eu(\omega)\,dP(\omega)
\]
In particular, the utility associated with elementary event $a\in\Omega$ can be defined as $u(a)=\int_\Omega u(\omega)\,\delta_a(\omega)$, where $\delta_a$ is the elementary probability measure (i.e. the Dirac $\delta$-measure concentrated entirely on $a\in\Omega$).

Probability measures, or `lotteries', are elements of a vector space, such as, for example, the space $Y=\mathcal{M}_c(\Omega)$ of signed Radon measures on $\Omega$ \cite{Bourbaki63}.  We remind that signed Radon measures are bounded linear functionals $y(f)=\int f\,dy$ on the space $X=\mathcal{C}_c(\Omega,\bR)$ of continuous functions $f:\Omega\to\bR$ with compact support (i.e. $Y=X'$ is the space of distributions dual of the space $X$ of test functions).  Measures that are non-negative $y(E)=\int_Edy\geq0$ for all $E\subseteq\Omega$ form a convex cone in $Y$.  The normalization condition $y(\Omega)=1$ defines an affine set in $Y$, and its intersection with the positive cone defines its base:
\[
\mathcal{P}(\Omega):=\{y\in Y:y(E)\geq0\,,\ y(\Omega)=1\}
\]
The base $\cP(\Omega)$ is the set of all Radon probability measures on $\Omega$.  It is a weakly compact convex set, and by the Krein-Milman theorem each point $p\in\mathcal{P}(\Omega)$ can be represented as a convex combination of its extreme points $\delta_\omega$ --- the elementary measures on $\Omega$.  In fact, $\mathcal{P}(\Omega)$ is a simplex, so that representations are unique, and the set $\ext\,\mathcal{P}(\Omega)$ of extreme points is identified with the set $\Omega$ of elementary events.  Figure~\ref{fig:linear} shows an example of $2$-simplex, which is the set $\cP(\Omega)$ of lotteries over three outcomes $\Omega=\{\omega_1,\omega_2,\omega_3\}$.

A question that arises in this construction is: How should the preference relation $\lesssim$ on $\Omega$ be extended to the set $\cP(\Omega)$ of all `lotteries' over $\Omega$?  Because $\cP(\Omega)$ is a subset of a vector space, it is quite natural to require that $\lesssim$ satisfies axioms~(\ref{ax:lambda}), (\ref{ax:addition}) and (\ref{ax:Archimedian}), and this leads immediately to the following result.

\begin{theorem}[Expected Utility]
A totally pre-ordered vector space $(Y,\lesssim)$ satisfies axioms~(\ref{ax:lambda}), (\ref{ax:addition}) and (\ref{ax:Archimedian}) if and only if $(Y,\lesssim)$ has a utility representation by a closed\footnote{We use the notion of a closed functional, because the topology in $Y$ is not defined.} linear functional $u:Y\to\bR$.
\label{th:linear-utility}
\end{theorem}

\begin{proof}
($\Leftarrow$)
The necessity of axioms~(\ref{ax:lambda}), (\ref{ax:addition}) follows immediately from linearity of functional $u:Y\to\bR$, representing $(Y,\lesssim)$.  The Archimedian axiom~(\ref{ax:Archimedian}) is necessary if $u$ is closed: $u(x)=\upsilon$ for every convergent sequence $x_n\to x$ such that $u(x_n)\to\upsilon$ (i.e. $u(\lim x_n)=\lim u(x_n)$).  Indeed, assume $nz\lesssim y$ for all $n\in\bN$ and some $z>0$.  Then $x_n=y/n\gtrsim z>0$ for all $n\in\bN$, and therefore $\lim u(x_n)\geq u(z)>0$, because $u$ is a representation of $(Y,\lesssim)$.  But $\lim x_n=y\lim(1/n)=y\cdot0=0$, meaning that $u$ is not closed.

($\Rightarrow$)
First, we show that axioms~(\ref{ax:lambda}) and (\ref{ax:addition}) imply that the equivalence classes $[x]:=\{y:x\sim y\}$ are affine.  Indeed, assume they are not affine.  Then there exist two points $x$, $y$ in $[x]$ such that the line passing through them contains a point that does not belong to $[x]$.  That is $(1-\lambda)x+\lambda y\notin[x]$ for some $\lambda\in\bR$ and $x$, $y\in[x]$.  This means, for example, that
\[
x\sim y<(1-\lambda)x+\lambda y
\]
Using property~(\ref{eq:equivalence}), let us replace $\lambda y$ by the equivalent $\lambda x$, so that we have
\[
x\sim y<(1-\lambda)x+\lambda x=x
\]
But $y<x$ contradicts our assumption $x\sim y$ (and $x<x$ is a contradiction as well).  Therefore, for any $x$ and $y$ in $[x]$ the whole line $\{z:z=(1-\lambda)x+\lambda y\,,\ \lambda\in\bR\}$ is also in $[x]$.  Thus, if $(Y,\lesssim)$ has a utility representation, then it can be taken to be an affine or a linear functional $u(y)=\int u\,dy$, because it must have affine level sets $[x]=\{y:u(y)=u(x)\}$.\footnote{An affine functional $h$ and a linear functional $u(y)=h(y)-h(0)$ have isomorphic level sets.}

Second, we prove that axiom~(\ref{ax:Archimedian}) implies that there is a countable order-dense subset $Q\subset Y$, so that $(Y,\lesssim)$ has a utility representation by Theorem~\ref{th:order-dense}.  Indeed, take $Q:=\{mz/n:z>0,\ m/n\in\mathbb{Q}\}$.  Case $x<0<y$ is trivial, therefore consider the case $0<x<y$ (or equivalently $x<y<0$).  Because $z>0$, axiom~(\ref{ax:Archimedian}) implies that $z/n<y-x$ for some $n\in\bN$ or
\[
x<z/n+x<y
\]
If $x\notin Q$, then $x\nsim mz/n\in Q$ for all $m/n\in\mathbb{Q}$, or $mz/n<x<(m+1)z/n$ for some $m$, $n\in \bN$.  But this means $(m+1)z/n<y$, because $(m+1)z/n=z/n+mz/n<z/n+x<y$.  Thus, we have found $q=(m+1)z/n\in Q$ with the property $x<q<y$.\qed
\end{proof}

The restriction of the linear functional $u(y)=\int u\,dy$ to the set $\cP(\Omega)$ of probability measures is the expected utility: $u(y)|_\cP=\bE_P\{u\}=\int u\,dP$.  Thus, Theorem~\ref{th:linear-utility} generalizes the expected utility representation of preference relation $(\cP(\Omega),\lesssim)$ satisfying axioms~(\ref{ax:lambda}), (\ref{ax:addition}) and (\ref{ax:Archimedian}) (the EU postulate \cite{Neumann-Morgenstern}):
\begin{equation}
Q\lesssim P\quad\iff\quad\bE_Q\{u\}\leq\bE_P\{u\}\qquad\forall\,Q,P\in\cP(\Omega)
\label{eq:eu-postulate}
\end{equation}
Note that the direct proof of the above result can be quite complicated (e.g. it spans five pages in \cite{Kreps88}).  The proof using Theorem~\ref{th:linear-utility} appears to be simpler.

\section{Violations of Linearity and Asymmetry of Risk}
\label{sec:risk}

The linear theory described above is quite beautiful, because it follows naturally from some basic mathematical principles.  However, its final conclusion, the EU postulate~(\ref{eq:eu-postulate}), appears to be over-simplistic: according to it, a decision-maker should pay attention only to the first moments of utility distributions; all other information, such as their variance or higher order statistics, should be disregarded.  The fact that this idea is rather naive becomes obvious, when one attempts to apply it in practical situations involving money.  Many counter-examples and paradoxes have been discussed in the literature (e.g. see \cite{Allais53,Ellsberg61,Tversky81}).  Here we review some of them with the aim to show that the expected utility does not fully characterize an important aspect of decision-making under uncertainty, and that is the concept of risk.

\subsection{Risk Aversion}
\label{sec:risk-aversion}

Consider the following example:

\begin{example}
Let $\Omega=\{\omega_1,\ldots,\omega_4\}$ be four elementary outcomes that carry utilities $u(\omega)\in\{-\$ 1000,-\$ 1,\$ 1,\$ 1000\}$.  Consider two lotteries over these outcomes:
\[
P(\omega)\in\{0,.5,.5,0\}\,,\quad Q(\omega)\in\{.5,0,0,.5\}
\]
Both lotteries have zero expected utility $\bE_P\{u\}=\bE_Q\{u\}=\$ 0$.  Thus, according to the EU postulate~(\ref{eq:eu-postulate}), a rational agent should be indifferent $P\sim Q$.  However, lottery $Q$ appears to be more `risky', as there is an equal chance of loosing or winning $\$1000$ in $Q$ as opposed to loosing or winning just $\$1$.  Thus, a risk-averse agent should prefer $P>Q$.
\end{example}

This example illustrates that risk is related somehow to the higher order moments of utility distribution, such as variance $\sigma^2(u)$ (i.e. expected squared deviation from the mean).  In fact, financial risk is often defined as the probability of an outcome that is preferred much less than the expected outcome (i.e. the probability of negative deviation $u(\omega)-\bE_P\{u\}<0$).  Other higher order statistics can also be useful, and in the next section we discuss entropy and information in relation to risk.  The following example supports this idea.

\begin{example}[The Ellsberg paradox \cite{Ellsberg61}]
The lotteries $P$ and $Q$ are represented by two urns with 100 balls each.  There are 50 red and 50 white balls in urn $P$; the ratio of red and while balls in urn $Q$ is unknown.  The player is offered to draw a ball from any of the two urns.  If the ball is red, then the player wins $\$100$.  Which of the urns should the player prefer?

The choice can be represented by two lotteries:
\begin{list}{}{}
\item [$P$]: The probabilities of winning $\$100$ and winning nothing are equal: $P(\$100)=P(\$0)=.5$.
\item [$Q$]: The probability of winning $\$100$ is unknown: $Q(\$100)=t\in[0,1]$.
\end{list}
One can check that $\bE_P\{u\}=\bE_Q\{u\}=\int_0^1 \left(\$100\cdot t+\$0\cdot (1-t)\right)\,dt=\$50$.  Thus, the player should be indifferent $P\sim Q$ according to the EU postulate~(\ref{eq:eu-postulate}).  There is an overwhelming evidence, however, that most humans prefer $P>Q$, which suggests that they prefer more information about the parameters of the distribution in this game.
\end{example}

Whether an agent is risk-averse or not may depend on its wealth.  However, it is generally assumed that most rational agents are risk-averse, when unusually high amounts of money are involved, and this is represented by a concave `utility of money' function \cite{Kreps88}.  This is justified by the idea that the utility of gaining $\$1$ relative to some amount $C>0$ is decreasing as $C$ grows.  The origin of this idea is in the St. Petersburg paradox due to Nicolas Bernoulli (1713).

\begin{example}[The St. Petersburg lottery]
The lottery is played by tossing a fair coin repeatedly until the first head appears.  Thus, the set $\Omega$ of elementary events is the set of all sequences of $n\in\bN$ coin tosses.  If the head appeared on the $n$th toss, then the player wins $\$ 2^n$.  Clearly, it is impossible to loose in this lottery.  However, to play the lottery the player must pay an entree fee $C>0$.  The question is: What amount $C>0$ should a rational agent pay?

According to the EU postulate~(\ref{eq:eu-postulate}) the fee $C$ should not exceed the expected utility $\bE_P\{u\}$ of the lottery.  It is easy to see, however, that for a fair coin $P(\omega_n)=2^{-n}$, and therefore the expected utility diverges
\[
\bE_P\{u\}=\sum_{n=1}^\infty \frac{2^n}{2^n}
\]
Thus, any amount $C>0$ appears to be a rational fee to pay.  The paradox is that not many people would pay more than $C=\$2$.  The solution proposed by Daniel Bernoulli in \cite{Bernoulli1738} was to convert the utility $2^n\mapsto\log_22^n=n$.  Although this does not resolve the general problem of unbounded expectations (e.g. one can introduce another lottery $Q$ such that $\bE_Q\{\log_2(u)\}$ diverges), this was the first example of a concave function used to represent risk-averse utility.
\end{example}

Note that although the `utility of money' can be concave as a function of $x(\omega)\in\bR$ amount, the expected utility is still a linear functional on the set $\cP(\Omega)$ of lotteries.  The level sets of the expected utility are affine sets corresponding to equivalence classes of lotteries with respect to $\lesssim$ that are parallel to each other (see Figure~\ref{fig:linear}).  The risk-averse concave modification simply gives less weight to higher values $x(\omega)$.  This modification also reduces the variance of the lottery.

\subsection{Risk Taking}
\label{sec:risk-taking}

It is not difficult to introduce a lottery in which risk-taking appears to be rational.

\begin{example}[The `Northern Rock' lottery]
\label{ex:northern-rock}
A player is allowed to borrow any amount $C>0$ from a bank.  When repayment is due, the amount to repay is decided in the St. Petersburg lottery: a fair coin is tossed repeatedly until the first head appears.  If the head appeared on the $n$th toss, then the player has to repay $\$ 2^n$ to the bank (i.e. the utility is $u(\omega_n)=-\$2^n$).  The question is: What amount $C>0$ should a rational agent borrow?

Again, according to the EU postulate~(\ref{eq:eu-postulate}) one should not borrow an amount $C$ that is less than the expected repayment $\bE_P\{-u\}$.  However, assuming that the probability $P(\omega_n)=2^{-n}$ for a fair coin, it is easy to see that the expected repayment diverges, and therefore a rational agent should not borrow at all.  Although the author did not conduct a systematic study of this problem, anecdotal evidence suggests that many people do borrow substantial amounts.  The solution to this paradox can be made similar to \cite{Bernoulli1738} by modifying the utility $-2^n\mapsto-\log_22^n=-n$.  Observe that the utility for repayments is not concave, but convex (negative logarithm), and therefore it appears to represent not risk-averse, but a risk-taking utility.
\end{example}

One of the most striking counter-examples to the expected utility postulate was introduced by Allais \cite{Allais53}.  Similar problems were studied by psychologists \cite{Tversky81}, which demonstrated the importance of how the outcomes are `framed' or perceived by an agent.  There are many versions of this problem, and the version below was used by the author in multiple talks on the subject.

\begin{example}[The Allais paradox \cite{Allais53}]
\label{ex:allais}
Consider which of the two lotteries you prefer to play:
\begin{list}{}{}
\item [$P$]: The player can win $\$300$ with probability $P(\$300)=1/3$, or win nothing with $P(\$0)=2/3$.
\item [$Q$]: The player wins $\$100$ with certainty $Q(\$100)=1$.
\end{list}
One can check that $\bE_P\{u\}=\bE_Q\{u\}=\$100$, which implies indifference $P\sim Q$ according to the EU postulate~(\ref{eq:eu-postulate}).  There is an overwhelming evidence, however, that most humans prefer $P<Q$, which suggests that they are risk-averse in this game.  Consider now another set of two lotteries:
\begin{list}{}{}
\item [$P$]: The player looses $\$300$ with probability $P(-\$300)=1/3$, or looses nothing with $P(\$0)=2/3$.
\item [$Q$]: The player looses $\$100$ with certainty $Q(-\$100)=1$.
\end{list}
Again, it is easy to check that $\bE_P\{u\}=\bE_Q\{u\}=-\$100$, corresponding to indifference $P\sim Q$ according to the EU postulate~(\ref{eq:eu-postulate}).  However, most humans prefer $P>Q$, which suggests a risk-taking behaviour.
\end{example}

This phenomenon of switching from risk-averse to risk-taking behaviour (also referred to as the `reflection effect') is observed in a number of other similar problems (e.g. see \cite{Tversky81}).  Specifically, a risk-averse preference is observed when the outcomes are associated with gains (positive change of utility), while a risk-taking preference is observed when the outcomes are associated with losses.  Note that gains can be converted into losses by multiplying their utility by $-1$ and vice versa.  In fact, this reflection was used in the construction of Example~\ref{ex:northern-rock} from the St. Petersburg lottery.  The use of concave functions for a risk-averse utility and convex functions for a risk-taking utility can also be explained using this reflection: recall that function $u(x)$ is concave if and only if $-u(x)$ is convex.

\subsection{Why is This a Paradox?}

The switch from a risk-averse pattern for gains to a risk-taking pattern for losses is quite systematic \cite{Tversky81,Loomes-Sugden82:_regret}, and the Allais paradox was demonstrated in numerous experiments \cite{Huck-Muller12} including professional traders \cite{List-Haigh05}.  This asymmetric perception of risk has been modelled in prospect theory \cite{Kahneman-Tversky79} by an $S$-shaped value function, such as a function shown on Figure~\ref{fig:ui-s}, which has a concave branch for outcomes associated with gains and convex branch for outcomes associated with losses.  Although this descriptive theory has gained significant recognition among psychologists and behavioural economists, the proposed asymmetry of risk appears to violate the beautiful and natural set of axioms behind the expected utility postulate \cite{Neumann-Morgenstern} (specifically, axioms~(\ref{ax:lambda}) and (\ref{ax:addition})).

\begin{figure}[!ht]
\centering
\input{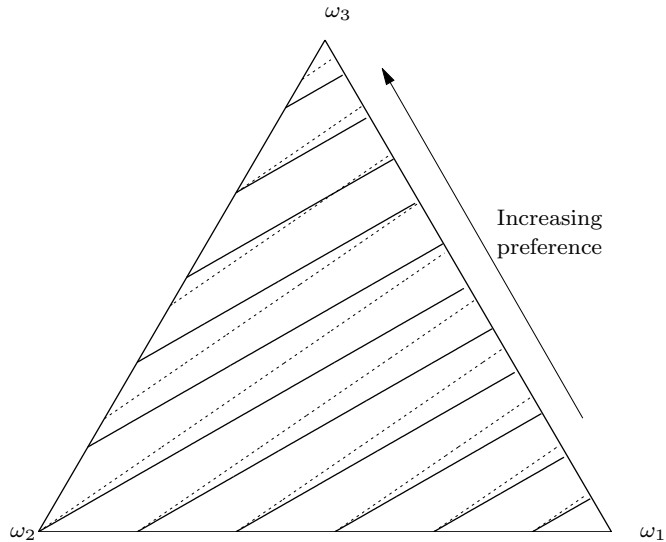}
\caption{Level sets of expected utility on the $2$-simplex of probability measures over set $\Omega=\{\omega_1,\omega_2,\omega_3\}$ with preference $\omega_1<\omega_2<\omega_3$.  Dotted lines represent level sets after a risk-averse modification of the utility function.}
\label{fig:linear}
\end{figure}

As mentioned earlier, the expected utility $\bE_P\{u\}=\int u\,dP$ is a linear functional on the set $\cP(\Omega)$ of probability measures (lotteries) regardless of the `shape' of the utility function $u:\Omega\to\bR$ on the extreme points $\ext\,\cP(\Omega)\equiv\Omega$.  The equivalence classes of the preference relation $\lesssim$ induced on the set of lotteries $\cP(\Omega)$ by the expected utility are the level sets $[\upsilon]:=\{P:\bE_P\{u\}=\upsilon\}$, and they are affine sets.  These level sets are shown on Figure~\ref{fig:linear} by parallel lines, where the triangle (a $2$-simplex) represents the set $\cP(\Omega)$ of lotteries over three elements.  Assuming the preference relation $\omega_1<\omega_2<\omega_3$, and taking the utility of $\omega_2$ as the reference level, lotteries above the reference level set can be considered as gains, while lotteries below the reference as losses.   To model a risk-averse pattern, one has to give lower weight to the outcomes with higher utility.  This change is shown on Figure~\ref{fig:linear} by dotted parallel lines depicting the new level sets.  One can notice that lotteries with higher variances or entropies (these are lotteries closer to the middle point of the simplex) are preferred less than they were before the `risk-averse' modification of utility (they are below the dotted lines).  However, because the level sets are parallel to each other, this change applies equally to gains and losses (i.e. lotteries above and below the reference level).  Thus, if a rational agent uses the expected utility model to rank lotteries, then he only can be risk-averse or risk-taking, but not both.  This observation was illustrated on a $2$-simplex in \cite{Machina82}, and it clearly shows why the expected utility theory alone cannot explain the switch from risk-averse to risk-taking behaviour observed in many examples discussed above.  Thus, it appears that human decision-makers violate the linear axioms~(\ref{ax:lambda}) and (\ref{ax:addition}), and several `non-expected' utility theories have been proposed, such as the regret theory \cite{Loomes-Sugden82:_regret} (see \cite{Starmer2000,Machina03,Machina04} for a review of many others).

\section{Risk and Value of Information}
\label{sec:u-of-i}

As discussed previously, risk is related to a deviation from expected utility, and many examples suggest its relation to variance or higher order statistics of the utility distribution.  Another functional characterizing the distribution is \emph{entropy}, which is closely related to variance and higher order cumulants of a random variable.  Entropy defines the maximum amount of information that a random variable can communicate.  Although information is measured in \emph{bits} or \emph{nats} that have no monetary value, when put in the context of decision-making or estimation, information defines the upper and lower bounds of the expected utility.  This amalgamation of expected utility and information is known as the \emph{value of information} theory pioneered by \cite{Stratonovich65}.  Remarkably, the value of information function has two distinct branches --- one is concave, representing the upper frontier of expected utility, while another is convex, representing the lower frontier of expected utility.  Interestingly, it was shown recently that these geometric properties do not depend on the definition of information itself, but follow from linearity of the expected utility \cite{Belavkin11:_optim}.  In this section, we discuss the classical notion of value of information, its generalization, and how it can be related to asymmetry of risk.

\subsection{Information and Entropy}

Information measures the ability of two or more systems to communicate, and therefore depend on each other.  System $A$ influences system $B$ (or $B$ depends on $A$) if the conditional probability $P(B\mid A)$ is different from the prior probability $P(B)$; or equivalently, if the joint probability $P(A\cap B)$ is different from the product probability $Q(A)\otimes P(B)$ of the marginals.  Shannon defined \emph{mutual information} \cite{Shannon48} as the expectation of the logarithmic difference of these probabilities:
\[
I_S(A,B)
:=\int_{A\times B}\left[\ln\frac{dP(b\mid a)}{dP(b)}\right]\,dP(a,b)
%=\int_{A\times B}\left[\ln\frac{dP(a,b)}{dQ(a)\,dP(b)}\right]\,dP(a,b)
\]
Mutual information is always non-negative with $I_S(A,B)=0$ if and only if $A$ and $B$ are independent (i.e. $P(B\mid A)=P(B)$).  The supremum of $I_S(A,B)$ is attained for $P(B\mid A)$ corresponding to an injective mapping $f:A\to B$, and it can be infinite.  Note that mutual information in this case equals the \emph{entropy} of the marginal distributions.

Indeed, recall that entropy of distribution $P(B)$ is defined as the expectation of its negative logarithm:
\[
H(B):=-\int_B[\ln dP(b)]\,dP(b)
\]
One can rewrite the definition of mutual information as the difference of marginal and conditional entropies:
\[
I_S(A,B)=H(B)-H(B\mid A)=H(A)-H(A\mid B)
\]
When $P(B\mid A)$ corresponds to a function $f:A\to B$, the conditional entropy is zero $H(B\mid A)=0$, and the mutual information equals entropy $H(B)$.  For example, by considering $A\equiv B$ one can define entropy as \emph{self-information} $I_S(B,B)=H(B)-H(B\mid B)=H(B)$ (i.e. $P(B\mid B)$ is the identity mapping $\mathrm{id}:B\to B$).  More generally, conditional entropies are zero for any bijection $f:A\to B$, so that $I_S(A,B)=H(A)=H(B)$ is the supremum of $I_S(A,B)$.  Thus, we can give the following variational definition of entropy:
\[
H(B)=I_S(B,B)=\sup_{P(A\cap B)}\left\{I_S(A,B):\int_A dP(B\mid a)\,dQ(a)=P(B)\right\}
\]
where the supremum is taken over all joint probability measures $P(A\cap B)$ such that $P(B)$ is its marginal.  This definition shows that entropy $H(B)$ is an information \emph{potential}, because it represents the maximum information that system $B$ with distribution $P(B)$ can communicate about another system.  In this context, it is called Boltzmann information, and its supremum $\sup H(B)=\ln|B|$ is called Hartley information.

The relation of entropy to information may help in the analysis of choice under uncertainty.  Indeed, lotteries with higher entropy have greater information potential.  Thus, although lotteries $P$ and $Q$ in Example~\ref{ex:allais} have the same expected utilities, their entropies or information potentials are very different.  In fact, because lottery $Q$ in Example~\ref{ex:allais} offers a fixed amount of money with certainty, its entropy is zero.  Information may be useful to a decision-maker and therefore may also carry a utility.

\subsection{Classical Value of Information}

The idea that information may improve the performance of statistical estimation and control systems was developed into a rigorous theory in the mid 1960-es by Stratonovich and Grishanin \cite{Stratonovich65,Stratonovich-Grishanin66,Stratonovich66:_value_automata,Grishanin-Stratonovich66,Stratonovich67:_inf_dyn,Stratonovich-Grishanin68,Stratonovich75:_inf}.  Consider a composite system $A\times B$ with joint distribution $P(A\cap B)=P(B\mid A)\otimes Q(A)$ and a utility function $u:A\times B\to\bR$.  For example, $A$ may represent a system to be estimated or controlled, $B$ may represent an estimator or a controller, and $u(a,b)$ measures the quality of estimation or control (e.g. a negative error).  In game theory, $A\times B$ may represent the set of pure strategies of two players, and $u(a,b)$ a reward function to player $B$.  If there is no information communicated between $A$ and $B$, then the expected utility $\bE_P\{u(a,b)\}$ can be maximized in a standard way by choosing elements $b\in B$ based on the distribution $Q(A)$.  On the other hand, if there is complete information (i.e. $a\in A$ is known or observed), then $u(a,b)$ can be maximized by choosing $b\in B$ for each $a\in A$.  The \emph{value of Shannon's information amount} $\lambda$ (or \emph{$\lambda$-information}) was defined as the maximum expected utility that can be achieved subject to the constraint that mutual information $I_S(A,B)$ does not exceed $\lambda$:
\[
\overline u_S(\lambda):=\sup_{P(B\mid A)}\left\{\bE_P\{u(a,b)\}:I_S(A,B)\leq\lambda\right\}
\]
Note that the expected utility and mutual information above are computed using the joint distributions $P(A\cap B)=P(B\mid A)\otimes Q(A)$, while the maximization is over the conditional probabilities $P(B\mid A)$ with the marginal distribution $Q(A)$ considered to be fixed.  The subscript in $\overline u_S(\lambda)$ denotes that it is the value of information of Shannon type.  Stratonovich also defined the value of information $\overline u_B(\lambda)$ of Boltzmann type, in which maximization is done with the additional constraint that $P(B\mid A)$ must be a function $f:A\to B$ such that the entropy $H(B)=H(f(A))\leq\lambda$, and value of information $\overline u_H(\lambda)$ of Hartley type with the constraint on cardinality $\ln|f(A)|\leq\lambda$ \cite{Stratonovich75:_inf}.  Stratonovich also showed the inequality $\overline u_S(\lambda)\geq\overline u_B(\lambda)\geq\overline u_H(\lambda)$, which follows from the fact that $I_S(A,B)\leq H(f(A))\leq\ln|f(A)|$, and proved a theorem about asymptotic equivalence of all types of $\lambda$-information (Theorems~11.1--2 in~\cite{Stratonovich75:_inf}).

The function $\overline u_S(\lambda)$ defines the upper frontier of the expected utility.  One may also be interested in the lower frontier (i.e. the worst case scenario) defined similarly using minimization:
\[
\underline u_S(\lambda):=\inf_{P(B\mid A)}\left\{\bE_P\{u(a,b)\}:I_S(A,B)\leq\lambda\right\}
\]
Functions $\overline u_S(\lambda)$ and $\underline u_S(\lambda)$ were referred to in \cite{Stratonovich75:_inf} as \emph{normal} and \emph{abnormal} branches of $\lambda$-information, representing respectively the maximal gain $\overline u_S(\lambda)-\overline u_S(0)\geq0$ and the maximal loss $\underline u_S(\lambda)-\underline u_S(0)\leq0$.  Observe that $\underline u_S(\lambda)=-\overline{(-u)}_S(\lambda)$\footnote{Note that $\underline u_S(\lambda)\neq-\overline{u}_S(\lambda)$ in general, and one of the branches may be empty.} (because $\inf u=-\sup(-u)$), which uses the reflection $u(x)\mapsto -u(x)$ to switch between gains and losses, as discussed in Section~\ref{sec:risk-taking} (Example~\ref{ex:allais}).  It was shown in \cite{Stratonovich75:_inf} that the normal branch $\overline u_S(\lambda)$ is concave and non-decreasing, while abnormal branch $\underline u_S(\lambda)$ is convex and non-increasing.  These properties can be used to give the following information-theoretic interpretation of humans' perception of risk.

Indeed, lotteries with non-zero entropy have a non-zero information potential, which means that after playing the lottery, information $\lambda$ may increase or decrease by the amount $\Delta\lambda$.  The value of this potential information, however, can be represented either by the normal branch $\overline u_S(\lambda)$, if lotteries are associated with gains, or by the abnormal branch $\underline u_S(\lambda)$, if lotteries are associated with losses.  Using the absolute value $|\lambda|$ in the constraint $I_S\leq|\lambda|$, one can plot the normal branch $\overline u_S(\lambda)$ against `positive' information $\lambda\geq0$, associated with gains, while the abnormal branch $\underline u_S(\lambda)$ against `negative' information $\lambda\leq0$, associated with losses.  The graph of the resulting function is shown on Figure~\ref{fig:ui-s}, and it is similar to the $S$-shaped value function in prospect theory \cite{Kahneman-Tversky79}, because $\overline u_S(\lambda)$ is concave and $\underline u_S(\lambda)$ is convex.  The normal branch implies risk-aversion in choices associated with gains, because the potential increase $\overline u_S(\lambda+\Delta\lambda)-\overline u_S(\lambda)$ associated with $\Delta\lambda$ is less than the potential decrease $\overline u_S(\lambda)-\overline u_S(\lambda-\Delta\lambda)$.  On the other hand, convexity of the abnormal branch $\underline u_S(\lambda)$ implies risk-taking in choices associated with losses, because the potential increase $\underline u_S(\lambda)-\underline u_S(\lambda+\Delta\lambda)$ is greater than potential decrease $\underline u_S(\lambda-\Delta\lambda)-\underline u_S(\lambda)$ (here, we assume $\lambda\leq0$ as on Figure~\ref{fig:ui-s}).

Unfortunately, this explanation may appear simply as a curious coincidence, because proofs that $\overline u_S(\lambda)$ is concave and $\underline u_S(\lambda)$ is convex are usually based on very specific assumptions about information, such as convexity and differentiability of Shannon's information $I_S(A,B)$ as a functional of probability measures.  It can be shown, however, that the discussed properties of $\lambda$-information hold in a more general setting, when information is understood more abstractly \cite{Belavkin11:_optim}, and they follow only from the linearity of the expected utility, that is from axioms~(\ref{ax:lambda}) and (\ref{ax:addition}).

\subsection{Value of Abstract Information}

In this section, we discuss generalizations of the concept of information and show that the corresponding value functions have concave and convex branches.  Recall that the definition of Shannon's information, as well as entropy, involves a very specific functional --- the Kullback-Leibler divergence $D_{KL}(P,Q)$ \cite{Kullback59}.  If $P$ and $Q$ are two probability measures defined on the same $\sigma$-ring $\mathcal{R}(\Omega)$ of subsets of $\Omega$, and $P$ is absolutely continuous with respect to $Q$, then KL-divergence of $Q$ from $P$ is the expectation $\bE_P\{\ln(P/Q)\}$:
\[
D_{KL}(P,Q):=\int_\Omega\left[\ln\frac{dP(\omega)}{dQ(\omega)}\right]\,dP(\omega)
\]
It plays the role of a distance between distributions, because $D_{KL}(P,Q)\geq0$ for all $P$, $Q$, and $D_{KL}(P,Q)=0$ if and only if $P=Q$, but it is not a metric (in general, symmetry and the triangle inequality do not hold).  The unique property of the KL-divergence is that it satisfies the axiom of additivity of information from independent sources \cite{Khinchin57}:
\[
%\begin{equation}
D_{KL}(P_1\otimes P_2,Q_1\otimes Q_2)=D_{KL}(P_1,Q_1)+D_{KL}(P_2,Q_2)
%\label{eq:additivity}
%\end{equation}
\]

One can see that Shannon's mutual information $I_S(A,B)$ is the KL-divergence of the prior distribution $P(B)$ from posterior $P(B\mid A)$ (or equivalently of the product $Q(A)\otimes P(B)$ of marginals from the joint distribution $P(A\cap B)$).  Entropy can be interpreted as negative KL-divergence $-D_{KL}(P,\mu)$ of some reference measure $\mu$ (e.g. the Lebesgue measure on $\Omega$) from $P$.  One way to to generalize the notion of information is to consider other information distances.

By a {\em distance} one understands a non-negative function $D:Y\times Y\to\bR\cup\{\infty\}$ such that $y=z$ implies $D(y,z)=0$.  When $D$ is restricted to the set $\cP(\Omega)\subset Y$ of probability measures, we refer to it as an {\em information distance}.  If a closed functional  $F:Y\to\bR\cup\{\infty\}$ is minimized at $y_0$, then the distance $D(y,y_0)$ can be defined by the non-negative difference $F(y)-F(y_0)$.  More generally, a distance associated with $F$ can be defined as follows:

\begin{definition}[$F$-information distance]
A restriction to $\cP(\Omega)\subset Y$ of function $D_F:Y\times Y\to\bR\cup\{\infty\}$ associated with a closed functional $F:Y\to\bR\cup\{\infty\}$ as follows:
\begin{equation}
D_F(y,z):=\inf\{F(y)-F(z)-x(y-z):x\in\partial F(z)\}
\label{eq:distance}
\end{equation}
where $\partial F(z):=\{x\in X:x(y-z)\leq F(y)-F(z),\,\forall y\in Y\}$ is subdifferential of $F$ at $z$.  We define $D_F(y,z)=\infty$ if $\partial F(z)=\varnothing$ or $F(y)=\infty$.
\end{definition}

It follows immediately from the definition of $\partial F(z)$ that $D_F(y,z)\geq0$.  We note also that the notion of subdifferential can be applied to a non-convex function $F$.  However, non-empty $\partial F(z)$ implies $F(z)<\infty$ and $F(z)=F^{\ast\ast}(z)$, $\partial F(z)=\partial F^{\ast\ast}(z)$ (\cite{Rockafellar74}, Theorem~12).  Generally, $F^{\ast\ast}\leq F$, so that $F(y)-F(z)\geq F^{\ast\ast}(y)-F^{\ast\ast}(z)$ if $\partial F(z)\neq\varnothing$.  If $F$ is G\^{a}teaux differentiable at $z$, then $\partial F(z)$ has a single element $x=\nabla F(z)$, called the {\em gradient} of $F$ at $z$.  One can see that distance~(\ref{eq:distance}) is a generalization of the Bregman divergence for the case of a non-convex and non-differentiable $F$.  The KL-divergence is a particular example of Bregman divergence associated with strictly convex and smooth functional $F(y)=\int(\ln y-1)\,dy$.  Thus, an information constraint can be understood geometrically as constraint $D_F(P,Q)\leq\lambda$ on some $F$-information distance, and the value of information has the following geometric interpretation.

% The KL-divergence of $P$ from elementary measure $\delta(E)$ of some
% event $E\subseteq\Omega$ is $D_{KL}(\delta,P)=-\ln P(E)$, and it is
% called {\em self-information} (or {\em surprise}) of $E$.  The
% expected value $\bE_P\{-\ln P\}$ of surprise is the {\em entropy} of
% $P$.  Alternatively,

Let $X$ and $Y$ be two linear spaces in duality, and let $x(y)=\int x\,dy$ be a linear functional on $Y$.  Recall that the \emph{support function} of set $C\subseteq Y$ is sublinear mapping $sC:X\to\bR\cup\{\infty\}$ defined as
\[
sC(x):=\sup\{x(y):y\in C\}
\]
Because expected utility $\bE_P\{u\}$ is the restriction to $\cP(\Omega)\subset Y$ of linear functional $u(y)=\int u\,dy$, the value of Shannon's mutual information $\overline u_S(\lambda)$ coincides with the support function $sC(\lambda)(u)$ of set $C(\lambda)\subseteq\cP(\Omega)$, defined by the information constraint $I_S(A,B)\leq\lambda$ and evaluated at $u\in X$ corresponding to the utility function $u:\Omega\to\bR$.  Another way to define subsets $C(\lambda)$ is based on the notion of information \emph{resource} \cite{Belavkin11:_optim}.

Let $\{C(\lambda)\}_{\lambda\in\bR}$ be a family of non-empty closed sets such that $C(\lambda_1)\subseteq C(\lambda_2)$ for any $\lambda_1\leq\lambda_2$.  Then the support function $sC(\lambda)(x)$ is non-decreasing for $\lambda$.  The union of all sets $C(\lambda)\times[\lambda,\infty)$ is the epigraph of some closed functional $F:Y\to\bR\cup\{\infty\}$.  In fact, this functional can be defined as $F(y)=\inf\{\lambda:y\in C(\lambda)\}$.  Then each closed set $C(\lambda)$ is a sublevel set $\{y:F(y)\leq\lambda\}$.

\begin{definition}[Information resource]
A restriction to $\cP(\Omega)\subset Y$ of a closed functional $F:Y\to\bR\cup\{\infty\}$.
\end{definition}

A generalized notion of the value of information is given by the support function of subsets $C(\lambda)\subseteq\cP(\Omega)$, defined by constraints either on $F$-information distance from some reference point or on an information resource:
\begin{eqnarray*}
\overline u(\lambda)&:=&\sup\{u(y):F(y)\leq\lambda\}\\
\underline u(\lambda)&:=&\inf\{u(y):F(y)\leq\lambda\}
\end{eqnarray*}
Properties of the above value functions were studied in \cite{Belavkin11:_optim}.  In particular, the following result was proven (Proposition~3, \cite{Belavkin11:_optim}).

\begin{theorem}
Function $\overline u(\lambda)$ is strictly increasing and concave.  Function $\underline u(\lambda)$ is strictly decreasing and convex.
\label{th:concave-convex}
\end{theorem}

Here we outline the proof assuming the reader has some knowledge of convex analysis (e.g. see \cite{Rockafellar74,Tikhomirov90:_convex} for references).  See Propositions~1--3 in \cite{Belavkin11:_optim} for more details.

\begin{proof}
Variational problem $\overline u(\lambda)=\sup\{u(y):F(y)\leq\lambda\}$ is solved using the method of Lagrange multipliers.  The Lagrange function is
\[
K(y,\beta^{-1})=u(y)+\beta^{-1}[\lambda-F(y)]
\]
where $\beta^{-1}$ is a Lagrange multiplier associated with constraint $F(y)\leq\lambda$.  The necessary conditions of extremum of $K(y,\beta^{-1})$ are
\[
\bar y(\beta)\in\partial F^\ast(\beta u)\,,\qquad F(\bar y(\beta))=\lambda
\]
where $\partial F^\ast(x):=\{y\in Y:y(z-x)\leq F^\ast(z)-F^\ast(x),\,\forall z\in X\}$ is subdifferential of the dual functional $F^\ast(x)=\sup\{x(y)-F(y)\}$ (i.e. the Legendre-Fenchel transform of $F$).  If the convex closure $\cl\co\{y:F(y)\leq\lambda\}$ of sublevel set coincides with $\{y:F^{\ast\ast}(y)\leq\lambda\}$, then the above conditions are also sufficient.

The function $\beta^{-1}(\lambda)$ is the derivative $d\overline u(\lambda)/d\lambda$, because $\overline u(\lambda)=u(\bar y)+\beta^{-1}[\lambda-F(\bar y)]$.  Also, $\beta^{-1}=d\overline u(\lambda)/d\lambda\geq0$, because $\overline u(\lambda)$ is non-decreasing.  In fact, $\beta^{-1}=0$ if and only if $\lambda =\sup F(y)$, which implies that $\overline u(\lambda)$ is strictly increasing.

The fact that $\overline u(\lambda)$ is concave is proven by showing that its derivative $\beta^{-1}(\lambda)$ is non-increasing.  Consider two solutions $\bar y(\beta_1)$, $\bar y(\beta_2)$ for $\lambda_1\leq\lambda_2$.  Because $\bar y(\beta_i)\in\partial F^\ast(\beta_iu)$ and $\partial F^\ast$ is a monotone operator, we have
\[
(\beta_2-\beta_1)u(\bar y(\beta_2)-\bar y(\beta_1))\geq0
\]
The difference $u(\bar y(\beta_2)-\bar y(\beta_1))\geq0$, because $\overline u(\lambda)=u(\bar y(\beta))$ is non-decreasing.  Therefore, $\beta_2-\beta_1\geq0$, which proves that $\beta^{-1}(\lambda)$ is non-increasing.

The strictly decreasing and convex properties of $\underline u(\lambda)$ follow from the fact that $\underline u(\lambda)=-\overline{(-u)}(\lambda)$, and $\overline{(-u)}(\lambda)$ is strictly increasing and concave, as was shown above.\qed
\end{proof}

\section{Discussion}
\label{sec:discussion}

In this paper, we have reviewed mathematical arguments for the expected utility theory, and some behavioural arguments against it.  Our hope is that by the end of Section~\ref{sec:risk} the reader was sufficiently intrigued by the paradox following from the conflict between the logic and structure of the axiomatic theory of utility on one hand, and our own behaviour that appears to contradict it on the other.  Perhaps the key to this puzzle is in the fact that the contradiction occurs only when humans are presented with the problems, and humans are information seeking agents.  Our mind has evolved to learn and adapt to new information, and this suggests we need to take potential information (entropy) into account.

Analysis shows that the value of information is an $S$-shaped function, which mirrors some of the ideas of prospect theory \cite{Kahneman-Tversky79}, and therefore the value of information theory may explain humans' attitude to risk.  Unlike the descriptive nature of the value function for prospects, however, properties of the value of information are based on rigorous results.  Furthermore, because the value of information is defined as conditional extremum of expected utility, this normative theory does not contradict the axioms of expected utility.  Rather, it generalizes the von Neumann and Morgenstern theory by adding a non-linear component that reflects the agent's preferences about potential information.

\begin{acknowledgement}
This work was supported by UK EPSRC grant EP/H031936/1.
\end{acknowledgement}

%\ifthenelse{\boolean{final}}{The
%author thanks the referees for their comments which helped improve
%this paper.}{}

%spbasic spmpsci spphys

\bibliography{rvb,nn,other,newbib,ica,evolution}

\end{document}